\documentclass[11pt,a4paper]{amsart}

 \usepackage[utf8x]{inputenc}
\usepackage{latexsym}
\usepackage{color,graphicx,shortvrb}
\usepackage{amsmath, amssymb}
\usepackage{amsfonts}
\usepackage[colorlinks, bookmarks=true]{hyperref}

\newtheorem{theorem}{Theorem}[section]
\newtheorem{lemma}[theorem]{Lemma}
\newtheorem{proposition}[theorem]{Proposition}
\newtheorem{corollary}[theorem]{Corollary}

\theoremstyle{definition}

\author{J. M. Almira}
\title{On Popoviciu-Ionescu functional equation in one and several variables}

\begin{document}





\begin{abstract}
We study a functional equation first proposed by  T. Popoviciu \cite{P} in 1955. In the one-dimensional case, it was solved for the easiest case by Ionescu \cite{I} in 1956  and, for the general case, by  Ghiorcoiasiu and Roscau \cite{GR}  and Rad\'{o} \cite{R} in 1962. We present a solution to the equation both for the one-dimensional and the higher-dimensional cases, which is based on a generalization of Rad\'{o}'s theorem to distributions in a higher dimensional setting. For functions of a single variable, our solution is different from the existing ones and, for functions of several variables, our solution is, as far as we know, the first one that exists. The one-dimensional part already appeared in a previous paper by the author \cite{Aviejo}
\end{abstract}

\maketitle

\markboth{J. M. Almira}{Popoviciu-Ionescu functional equation}

\section{Motivation}
We consider the continuous solutions of the functional equation
\begin{equation}\label{Po}
\det\left[\begin{array}{cccccc}
f(x) & f(x+h) &  \cdots & f(x+nh) \\
f(x+h) & f(x+2h) & \cdots & f(x+(n+1)h)\\
\vdots & \vdots & \ \ddots & \vdots \\
f(x+nh) & f(x+(n+1)h) &  \cdots &  f(x+2nh)\\
\end{array} \right] =0 \text{ for all } x,h\in\mathbb{R}.
\end{equation}
This equation was proposed by T. Popoviciu \cite{P} for real functions of one real variable and was studied by several Romanian and Hungarian  mathematicians in the 1960's \cite{C,I,GR,R,S}. In particular, Iounescu \cite{I} solved it for continuous functions with $n=1,2$ and, later on, as a result of the joint efforts of Ghiorcoiasiu and Roscau \cite{GR} and Rad\'{o} \cite{R}, it was solved for arbitrary $n$. Concretely, Ghiorcoiasiu and Roscau proved that, if $f:\mathbb{R}\to\mathbb{R}$ is a continuous solution of \eqref{Po}, then there exist $H>0$ and continuous  functions $a_k:(0,H)\to \mathbb{R}$, $k=0,\cdots,n$, such that $(a_0(h),\cdots,a_n(h))\neq (0,\cdots,0)$ for some $h\in (0,H)$ and
\begin{equation}\label{rado}
a_0(h)f(x)+a_1(h)f(x+h)+\cdots +a_n(h)f(x+nh)=0 \text{ for all } x\in\mathbb{R} \text{ and all } 0\leq h<H. 
\end{equation}
and Rad\'{o} proved that, for continuous functions $f:\mathbb{R}\to\mathbb{R}$, the equation \eqref{rado} characterizes the exponential polynomials which solve an ordinary homogenous linear differential equation of order $n$, with constant coefficients, $A_0f+A_1f'+\cdots+A_nf^{(n)}=0$. 
%

For functional equations like \eqref{rado}, which can be viewed as depending on a parameter $h$, it makes sense to ask about the minimal sets of parameters  
$\{h_i\}_{i\in I}$ with the property that, if $f$ solves the equation with $h=h_i$ for all $i\in I$, then it solves the equation for all $h$. These kind of results are named Montel-type theorems after the seminal papers by the French mathematician Montel, who studied the problem for Fr\'{e}chet's functional equation $\Delta_h^{n}f=0$ \cite{M1,M2,M3} (see also \cite{A1,A2,A3,A4}). 
 
We use, in this paper, Anselone-Korevaar's theorem \cite{anselone} to demonstrate a Montel-type theorem connected to Rad\'{o}'s functional equation \eqref{rado}, which we re-formulate for distributions defined on $\mathbb{R}^d$, and use the corresponding result to give a new proof of the fact that continuous solutions $f:\mathbb{R}^d\to\mathbb{R}$ of \eqref{Po} are exponential polynomials.  


\section{A characterization of exponential polynomials and Rad\'{o}'s theorem for higher dimensions}

\begin{theorem} \label{main} Let $h_1, h_2,\dots,h_s$  be such that they span a dense additive subgroup
of $\mathbb{R}^d$. Let $f$ be a distribution on $\mathbb{R}^d$ such that there exist natural
numbers $n_i$, $i=1,\dots,s$  satisfying
\begin{equation}\label{uno}
  \dim \mathbf{span} \{f, \tau_{h_i}(f),\cdots, (\tau_{h_i})^{n_i}(f)\} \leq n_i, \  \  i=1,\cdots,s.
\end{equation}
Then $f$ is, in distributional sense, a continuous exponential polynomial. In particular, $f$ is an ordinary function which is equal almost everywhere, in the Lebesgue measure, to an exponential polynomial.
\end{theorem}

\begin{proof} If $f=0$ we are done. Thus, we impose $f\neq 0$.  Let us assume, without loss of generality,  that $n_i$ is the smallest natural number satisfying
\eqref{uno}, and let 
$$W_i=\mathbf{span} \{f, \tau_{h_i}(f),\cdots, (\tau_{h_i})^{n_i-1}(f)\}.$$
Obviously, $\dim(W_i)=n_i$ (otherwise $n_i$ would not be minimal). Furthermore,
\eqref{uno}  implies that $\tau_{h_i}(W_i)\subseteq W_i$. Hence $\tau_{h_i}$ defines an automorphism on $W_i$, since $\tau_{h_i}$ is always injective and $\dim W_i<\infty$.   Consequently, $(\tau_{h_i})^m(W_i)= W_i $ for all integral numbers $m$. In  particular, for each $m\in\mathbb{Z}$ there exist numbers $\{a_{i,m,k}\}_{k=0}^{n_i-1}$ such that
\[
\tau_{h_i}^m(f)=\sum_{k=0}^{n_i-1}a_{i,m,k}(\tau_{h_i})^k(f).
\]
It follows that 
\begin{eqnarray*}
&\ & \tau_{m_1h_1+\cdots+m_sh_s}(f) = (\tau_{h_1})^{m_1}(\tau_{h_2})^{m_2}\cdots (\tau_{h_s})^{m_s}(f)\\
&\ & \  \  = (\tau_{h_1})^{m_1}\cdots(\tau_{h_{s-1}})^{m_{s-1}}\left (\sum_{k_s=0}^{n_s-1}a_{s,m_s,k_s}(\tau_{h_s})^{k_s}(f)\right)\\
&\ & \  \  = \sum_{k_s=0}^{n_s-1} a_{s,m_s,k_s} (\tau_{h_1})^{m_1}\cdots(\tau_{h_{s-1}})^{m_{s-1}} (\tau_{h_s})^{k_s}(f)\\
\end{eqnarray*}
and, repeating the argument $s$ times, we get
\[
\tau_{m_1h_1+\cdots+m_sh_s}(f) = \sum_{k_s=0}^{n_s-1} \cdots \sum_{k_{1}=0}^{n_{1}-1} 
\left(\prod_{i=0}^{s-1} a_{s-i,m_{s-i},k_{s-i}} \right) (\tau_{h_s})^{k_s}  \cdots (\tau_{h_1})^{k_1}(f)
\]
In other words, if we consider the space
$$W=\mathbf{span} \{f, (\tau_{h_1})^{a_1}(\tau_{h_2})^{a_2}\cdots (\tau_{h_s})^{a_s}(f): 0\leq a_i<n_i, i=1,2,\cdots,s\}, $$
then
$$\tau_{m_1h_1+\cdots+m_sh_s}(f)\in W \text{ for all } (m_1,\cdots,m_s) \in \mathbb{Z}^s.$$
Hence every translation of $f$ belongs to $W$, since $h_1,\cdots,h_s$ span a dense additive subgroup of $\mathbb{R}^d$ and $W$ is finite
dimensional. The proof ends by applying Anselone-Korevaar's theorem.  
\end{proof}

\begin{proposition}
Every open subset $V$ of  $\mathbb{R}^d$ contains a finite set of vectors $\{h_0,\cdots,h_s\}$ which span a dense subgroup of $\mathbb{R}^d$.
\end{proposition}

\begin{proof} A well known result by Kronecker states that $\mathbb{Z}^d+(\theta_1,\theta_2,\cdots,\theta_d)\mathbb{Z}$ is dense in $\mathbb{R}^d$ if and only if 
$\{1,\theta_1,\cdots,\theta_d\}$ is $\mathbb{Q}$-linearly independent (see \cite[Theorem 442, page 382]{HW}).  Of course, the same claim holds true for the subgroup $\frac{1}{N}(\mathbb{Z}^d+(\theta_1,\theta_2,\cdots,\theta_d)\mathbb{Z})$ for every $N>0$. Hence, every open neighborhood of $(0,0,\cdots,0)$ contains a finite set of vectors $\{h_0,\cdots,h_s\}$ which span a dense subgroup of $(\mathbb{R}^d,+)$.

Let  $V$ be any open subset of  $\mathbb{R}^d$ and let  $x_0\in V$ and $\varepsilon >0$ be such that $x_0+B_0(\varepsilon)\subseteq V$.  Take $\{h_1,\cdots,h_s\}\subseteq B_0(\varepsilon) $  such that $h_1\mathbb{Z}+\cdots+h_s\mathbb{Z}$ is  dense in $\mathbb{R}^d$. Then $\{x_0, x_0+h_1,\cdots, x_0+h_s\}\subset V$  spans a dense additive subgroup   $\mathbb{R}^d$. 
\end{proof}

\begin{corollary}[Rad\'{o}'s theorem for higher dimensions] \label{rado_d} Assume that $f:\mathbb{R}^d\to\mathbb{R}$ is a continuous solution of 
\begin{equation}\label{radod}
a_0(h)f(x)+a_1(h)f(x+h)+\cdots +a_n(h)f(x+nh)=0 \text{ for all } x\in\mathbb{R}^d \text{ and all } h\in U 
\end{equation}
for a certain open set $U\subseteq \mathbb{R}^d$ and certain continuous functions $a_k:U\to\mathbb{R}$ such that $a=(a_0,\cdots,a_n)$ does not   vanish identically. Then $f$ is an exponential polynomial in $d$ variables. 
\end{corollary}

\begin{proof}
Let $h_0\in U$ be such that $a(h_0)\neq (0,\cdots,0)$. The continuity of $a$ implies that $a(h)\neq (0,\cdots,0)$ for all $h\in V$ for a certain open set $V\subseteq U$. Let us take $\{h_1,\cdots, h_s\}\subset V$ spanning a dense subgroup $h_1\mathbb{Z}+\cdots+h_s\mathbb{Z}$ of $\mathbb{R}^d$. Then   \eqref{radod} implies that 
\begin{equation}\label{dos}
  \dim \mathbf{span} \{f, \tau_{h_i}(f),\cdots, (\tau_{h_i})^{n}(f)\} \leq n, \  \  i=1,\cdots,s,
\end{equation}
and Theorem \ref{main} implies that $f$ is equal almost everywhere to an exponential polynomial.  Hence $f$ itself is an exponential polynomial, since $f$ is continuous.
\end{proof}

\section{Solution of Popoviciu-Ionescu functional equation}
We are now ready for a proof of the following result:

\begin{theorem}[Solution of Popoviciu-Ionescu functional equation in one variable]
Let $f:\mathbb{R}\to\mathbb{R}$ be a continuous function which solves \eqref{Po} for all $x,h\in\mathbb{R}$. Then $f$ is an exponential polynomial.
\end{theorem}
\begin{proof}  Ghiorcoiasiu and Roscau's theorem  \cite{GR} guarantees that $f$ solves \eqref{rado} for some continuous function $a(h)=(a_0(h),\cdots,a_n(h))$ such that   $a(h_0)\neq (0,\cdots,0)$ for some $h_0\in (0,H)$. Indeed, they prove that we can impose $a_n(h)=1$ for all $h\in (0,H)$ (see \cite[Theorem 5]{GR}). The result follows just applying Corollary \ref{rado_d} to $f$.
\end{proof}

The proof above appeared by the first time in \cite{Aviejo}. Here, we also study the continuous solutions of Popoviciu-Ionescu functional equation in several variables:
\begin{equation}\label{Po_d}
\det\left[\begin{array}{cccccc}
f(x) & f(x+h) &  \cdots & f(x+nh) \\
f(x+h) & f(x+2h) & \cdots & f(x+(n+1)h)\\
\vdots & \vdots & \ \ddots & \vdots \\
f(x+nh) & f(x+(n+1)h) &  \cdots &  f(x+2nh)\\
\end{array} \right] =0 \text{ for all } x,h\in\mathbb{R}^d.
\end{equation}
The technique used by Ghiorcoiasiu and Roscau to reduce this equation to equation \eqref{rado} fails in this context, since the proof of Theorem 3 of their paper \cite{GR} strongly depends on the fact that all arguments of the function live in the very same line. 
%
%
%
%

On the other hand, the following holds true:

\begin{theorem} Every exponential polynomial $f:\mathbb{R}^d\to\mathbb{C}$ solves  the equation \eqref{Po_d} for all $n$ large enough. 
\end{theorem}

\begin{proof} If $f$ is an exponential polynomial then 
$\tau(f)=\mathbf{span}\{\tau_hf:h\in\mathbb{R}^d\}$ is a finite dimensional space. Hence, if $n=\dim \tau(f)$ and $h\in\mathbb{R}^d$, there exist coefficients $a_k(h)$ such that   $a_0(h)f+a_1(h)\tau_h(f)+\cdots+a_n(h)\tau_{nh}(f)$ vanishes identically and, henceforth, $f$ solves  \eqref{Po_d}.
\end{proof}

Moreover, we can demonstrate the following:

\begin{lemma} \label{pro} Assume that $f:\mathbb{R}^d\to\mathbb{C}$ is a continuous solution of \eqref{Po_d}. Then $f$, restricted to any line $L\subset \mathbb{R}^d$, defines an exponential polynomial.
\end{lemma}

\begin{proof} Given $x_0,h_0\in \mathbb{R}^d$, we set $F_{x_0,h_0}(t)=f(x_0+th_0)$. Then $F_{x_0,h_0}$ is a continuous solution of  \eqref{Po}, so that it is an exponential polynomial.  
\end{proof}

For polynomial functions,  Prager and Schwaiger demonstrated in 2009 \cite[Theorem 14]{PS} that  if  $K$ is a field and $ f : K^d\to K $ is an ordinary algebraic polynomial function separately in each variable (which means that for any $1\leq k\leq d$ and any point $(a_1,\cdots,a_{k-1},a_{k+1},\cdots,a_d)\in K^{d-1}$, the function $f(a_1,\cdots,a_{k-1},x_k,a_{k+1},\cdots,a_d)$ is an ordinary algebraic polynomial in $x_k$) then $f$ is an ordinary algebraic polynomial function in $d$ variables provided that $K$ is finite or uncountable. Furthermore, for every countable infinite field $K$ there exists a function $f:K^2\to K$ which is an ordinary algebraic polynomial function separately in each variable and is not a generalized polynomial in both variables jointly. Of course, the result does not assume continuity of $f$ nor any common upper bound for the degrees of the polynomials $f(a_1,\cdots,a_{k-1},x_k,a_{k+1},\cdots,a_d)$. 

On the other hand, a result of this type is false for exponential polynomials in two or more variables since, for each $d\geq 2$, the function $f(x_1,\cdots,x_d)=e^{x_1x_2\cdots x_d}$ is not an exponential polynomial and, on the other hand, for every point $(a_1,\cdots,a_{k-1},a_{k+1},\cdots,a_d)\in \mathbb{R}^{d-1}$, the function $f(a_1,\cdots,a_{k-1},x_k,a_{k+1},\cdots,a_d)$ is an exponential polynomial in $x_k$ (in fact, it is an exponential function). 

Fortunately,  the lines that appear in Lemma \ref{pro} include 
\[
\{(a_1,\cdots,a_{k-1},x_k,a_{k+1},\cdots,a_d)=
(a_1,\cdots,a_{k-1},0,a_{k+1},\cdots,a_d)+x_k (0,\cdots,0,1,0,\cdots,0): x_k\in\mathbb{R}\}
\] 
as a very particular case, so that we have some extra information that could be useful to characterize the solutions of \eqref{Po_d} for arbitrary $d$. 

The keystone to get such a result is an old result about exponential polynomials demonstrated by  A. L. Ronkin  in 1978 \cite{Ron} (see \cite{Ron1} for the details):

\begin{theorem}[Ronkin,1978] Assume that $f:\mathbb{R}^d\to\mathbb{C}$ is an exponential polynomial in each variable when the others are fixed. Then the function has the form 
\[
f(x_1,\cdots,x_d)=\sum_{k=1}^sP_k(x_1,x_2,\cdots,x_d)e^{\lambda_k(x_1,\cdots,x_d)},
\]
where the functions $P_k$ are algebraic polynomials and the functions $\lambda_k$ are polynomials linear in each variable. 
\end{theorem} 

We are now in conditions to state and demonstrate the following

\begin{theorem}[Solution of Popoviciu-Ionescu functional equation, several variables] \label{carRon}
Assume that  $f:\mathbb{R}^d\to\mathbb{C}$ is a continuous solution of \eqref{Po_d}. Then $f$ is an exponential polynomial.
\end{theorem}

\begin{proof} It follows from Theorem \ref{carRon} and Lemma \ref{pro} that $f$ has the form 
\[
f(x_1,\cdots,x_d)=\sum_{k=1}^sP_k(x_1,x_2,\cdots,x_d)e^{\lambda_k(x_1,\cdots,x_d)},
\]
where the functions $P_k$ are algebraic polynomials and the functions $\lambda_k$ are polynomials linear in each variable. We need to demonstrate that the total degree of $\lambda_k$ is $\leq 1$ for every $k$. Assume, on the contrary, that $\lambda_{k_0}$ has total degree $N>1$ for a certain $k_0$. Then 
\[
\lambda_{k_0}(x_1,\cdots,x_d)=\alpha_0+\sum_{s=1}^N\sum_{\#\{i_1,i_2,\cdots,i_s\}=s,1\leq i_k\leq d}c^s_{i_1,\cdots,i_s}x_{i_1}\cdots x_{i_s}
\]
with at least one of the coefficients of the monomials of degree $N$ being different from $0$. Assume, for example, that $c^N_{j_1,\cdots,j_N}\neq 0$ and set $h_0=(b_1,\cdots,b_d)$, where $b_{j_k}=1$ for $k=1,\cdots,N$ and $b_{j}=0$ if $j\not\in \{j_1,\cdots,j_N\}$. Then  $\lambda_{k_0}(th_0)$ is a polynomial in $t$ whose leading term is given by  $c^N_{j_1,\cdots,j_N}t^N$. Consequently, the function $F_{0,h_0}(t)=f(th_0)$ is not an exponential polynomial, which contradicts Lemma \ref{pro}. Hence the total degree of $\lambda_{k}$ is $\leq 1$ for all $k$ and $f$ is an exponential polynomial. This ends the proof. 
\end{proof}


\bigskip

\noindent J. M.~ALMIRA\\
Departamento de Ingenier\'{\i}a y Tecnolog\'{\i}a de Computadores,  Universidad de Murcia. \\
 30100 Murcia, SPAIN\\
e-mail: \texttt{jmalmira@um.es}\\[2ex]



\end{document}